\documentclass[11pt]{article}
\usepackage{latexsym}
\usepackage{epsfig,enumerate}
\usepackage{amsmath,amsthm,amssymb}
\usepackage{algorithm2e}
\usepackage{bbold}
\usepackage{hyperref}
\usepackage[doi=false,isbn=false,url=false,style=numeric,giveninits,maxbibnames=99]{biblatex}
\renewbibmacro{in:}{}
\allowdisplaybreaks

\usepackage[margin=1in]{geometry}

 \usepackage[usenames]{color}\usepackage{graphicx}

 \newcommand{\of}[1]{\left( #1 \right) }

 \newcommand{\set}[1]{\left\{ #1 \right\}}
 
 \newcommand{\tbf}[1]{\textbf{#1}}
 \newcommand{\bfrac}[2]{\of{\frac{#1}{#2}}}
 
 \renewcommand{\phi}{\varphi}
 \renewcommand{\L}{L}

 \newcommand{\tz}{\tilde{z}}
 \newcommand{\ttz}{\tilde{\tbf{z}}}
\DeclareMathOperator{\Op}{Op}
\DeclareMathOperator{\op}{op}

\newtheorem{theorem}{Theorem}[section]
\newtheorem{problem}[theorem]{Problem} 
 
\newtheorem{proposition}[theorem]{Proposition}

{\theoremstyle{definition}

 }

\title{Full Degree Spanning Trees in Random Regular Graphs}

\author{Sarah Acquaviva\thanks{ \url{acquavivas1@montclair.edu}}\qquad\qquad Deepak Bal\thanks{ \url{deepak.bal@montclair.edu}} \\
\quad \\
Department of Mathematics \\
Montclair State University \\
Montclair, NJ 07043 USA}

\date{}
\begin{document}
\maketitle

\begin{abstract}

We study the problem of maximizing the number of full degree vertices in a spanning tree $T$ of a graph $G$; that is, the number of vertices whose degree in $T$ equals its degree in $G$. In cubic graphs, this problem is equivalent to maximizing the number of leaves in $T$ and minimizing the size of a connected dominating set of $G$. 
We provide an algorithm which produces (w.h.p.) a tree with at least $0.4591n$ vertices of full degree (and also, leaves) when run on a random cubic graph. This improves the previously best known lower bound of $0.4146 n$. We also provide lower bounds on the number of full degree vertices in the random regular graph $G(n,r)$ for $r \le 10$.

\end{abstract}

\section{Introduction}


A vertex $v$ of graph $G$ is of \emph{full degree} in a spanning tree $T$ if $deg_T(v)=deg_G(v)$. Let $\varphi(G)$ be the maximum number of full degree vertices in any spanning tree of $G$. This problem, variously referred to as the \emph{full degree spanning tree} (FDST) problem \cite{bhatia} or the \emph{degree preserving spanning tree} (DPST) problem \cite{broersma}, was first studied by Lewinter \cite{lewinter} and subsequently studied by Pothof and Schut \cite{pothof-schut}, Broersma et al.\ \cite{broersma}, and Bhatia et al.\ \cite{bhatia}. The FDST problem has the following nice application to water distribution networks (for more details see \cite{bhatia,broersma, pothof-schut}). To measure the flow across each edge in a water network, one may place flow meters on the edges of a co-tree\footnote{A co-tree in a graph $G$ is a set of edges whose removal leaves a spanning tree in $G$.} and then infer the flow on the edges of the tree. Alternatively, one may place pressure meters (which are apparently much less expensive) on the end points of an edge and compute the flow across that edge. In this set up, we would only need to place pressure meters on vertices which are incident with co-tree edges. In other words, we would \emph{not} need to place a pressure meter on any vertex of full degree in the tree, and so maximizing the number of such vertices minimizes the cost of pressure meters.  For the most part, the FDST problem has been studied from a complexity theoretic and algorithmic viewpoint (see \cite{gaspers, gnw, lrs} for some fairly recent work). The problem is, in general, NP-Hard.

There are simple absolute bounds on $\phi(G)$ in terms of the maximum degree $\Delta = \Delta(G)$ and the minimum degree $\delta = \delta(G)$.
\begin{proposition}\label{prop:simple-bds}
For any connected graph $G$ on $n$ vertices, we have 
\[\frac{n}{\Delta(\Delta-1) +1}\le\varphi(G) \le \frac{n-2}{\delta-1}.\]
\end{proposition}
\begin{proof}
We first prove the lower bound. Let $G^2$ denote the square of $G$, i.e., the graph formed by joining all vertices at distance at most 2 in $G$. If $I$ is an independent set in $G^2$, then the vertices of $I$ have disjoint neighborhoods in $G$. Thus all the vertices of $I$ can be taken to be of full degree in a spanning tree of $G$. So \[\phi(G) \ge \alpha(G^2) \ge \frac{n}{\Delta(G^2) +1} \ge \frac{n}{\Delta(\Delta-1)+1}\]
where we have used the well-known bound $\alpha(G) \ge n/(\Delta+1)$ for any graph $G$. 

To prove the upper bound, we let $T$ be a spanning tree of $G$ and let $d_T(v)$ denote the degree of vertex $v$ in $T$. 
Let $X$ be the set of all full degree vertices in $T$. Since all full degree vertices have degree at least $\delta$ and all non-full degree vertices in $T$ have degree at least 1,
\begin{align*}
2(n-1) = \sum_{v} d_T(v) &= \sum_{v \in X} d_T(v) + \sum_{v \notin X} d_T(v)\\
&\geq |X|\cdot \delta + (n-|X|)\cdot 1\\
&= (\delta-1)|X| + n.
\end{align*}
Thus, 
\[|X| \leq \frac{n-2}{\delta-1}\]
\end{proof}

For a connected $r$-regular graph $G$, we then have the bounds $\frac{n}{r^2 - r + 1} \le \phi(G) \le \frac{n}{r-1}$. The upper bound is essentially tight as can be seen by taking $G = K_{r-1}\square C_{n/(r-1)}$. In this construction, one can take one vertex from each of $n/(r-1) - 2$ copies of $K_{r-1}$ to be of full degree and so $\phi(G)\ge n/(r-1) - 2$. See Section \ref{sec:conc} for a question about the tightness of the lower bound. 

In the case of cubic graphs (i.e. $3$-regular graphs), this problem has been studied before under a different guise. Let $G$ be a connected graph with $n$ vertices.
A \emph{dominating set} of vertices is a subset $S\subseteq V(G)$ such that every vertex in $G$ is either in $S$ or adjacent to a vertex in $S$. If $G[S]$ is connected, then $S$ is said to be a \emph{connected dominating set}. Let $\gamma_C(G)$ be the minimum number of vertices in a connected dominating set of $G$ and let $\lambda(G)$ be the maximum number of leaves in a spanning tree of $G$.
 As observed in \cite{cwy}, the parameters $\gamma_C$ and $\lambda$ satisfy
 \[\lambda(G) = n- \gamma_C(G).\] To see this, note that (i) given a spanning tree $T$ with leaf set $L$, we have that $V(T)\setminus L$ is a connected dominating set and (ii) given a connected dominating set $D$, there is a spanning tree of $G$ in which $V(G)\setminus D$ is the leaf set.
 These parameters have been very well studied (see e.g.\ \cite{cwy}, \cite{griggs}, \cite{storer}, and others). For every connected cubic graph $G$, Storer \cite{storer} found that $\lambda(G)\geq\lceil(n/4)+2\rceil$. Duckworth and Wormald \cite{duckworth-wormald} proved that $\gamma_C(G)\leq 2n/3+O(1)$ for cubic graphs $G$ with girth at least 5. Griggs et al \cite{griggs} found that $\lambda(G)\geq\lceil(n/3)+(4/3)\rceil$ for every connected cubic graph $G$ that does not have a subgraph isomorphic to $K_4$ with one edge removed. 
 
We have the following observation which shows that in connected cubic graphs, the problems of finding $\lambda, \varphi$ and $\gamma_C$ are all equivalent. 
\begin{proposition}
    
For any connected cubic graph $G$, we have $\lambda(G) = \phi(G)+2$. Furthermore, for any $r\ge 4$ and any connected $r$-regular graph $G$, we have $\lambda(G) \ge (r-2)\phi(G) +2$.
\end{proposition}
\begin{proof}
    Let $r\ge 3$ and let $G$ be a connected $r$-regular graph.
    Let $T$ be a spanning tree of $G$ and let $x_i$ be the number of vertices in $T$ of degree $i$ for $i\in [r]$.  Then 
    \begin{align*}2(n-1) =\sum_{v\in T}d_T(v) &= \sum_{i=1}^r ix_i\\
    &= 2\of{\sum_{i=1}^r x_i} - x_1 + \sum_{i=3}^r (i-2)x_i \\
    & = 2n - x_1 + \sum_{i=3}^r (i-2)x_i \\
    & \ge 2n - x_1 +(r-2)x_r
    \end{align*}
    where the last inequality is equality in the case $r=3$.
    Thus $x_1 \ge (r-2)x_r+2$ and the result follows.
\end{proof}




In this paper, we seek to understand the average behavior of $\phi$ for $r$-regular graphs by considering random regular graphs.  A \emph{random $r$-regular graph}, denoted $G(n,r)$, is a graph  chosen uniformly at random from all $r$-regular graphs on vertex set $[n] := \{1,2,\ldots, n\}$. For background on random regular graphs, see Wormald's survey \cite{wormald-rrg}.  
Duckworth \cite{duckworth} and Duckworth and Mans \cite{duckworth-mans} developed an algorithm for the connected dominating set problem and analyzed it on random $d$-regular graphs for $d$ fixed.
An event occurs \emph{with high probability (w.h.p.)}  if the probability that the event occurs tends towards 1 as $n$ tends towards $\infty$. Duckworth's  algorithm \cite{duckworth} produces a connected dominating set on the random cubic graph $G\sim G(n,3)$ that w.h.p.\ has size less than $0.5854n + o(n)$. Thus, w.h.p.\ his algorithm gives $\gamma_C(G)\le 0.5854n + o(n)$ and $\varphi(G)\ge 0.4146n + o(n)$.

We approach this problem by providing an algorithm that attempts to maximize the number of full degree vertices in spanning trees of random regular graphs. Somewhat surprisingly, our simple algorithm gives a decent improvement over the bounds given in \cite{duckworth, duckworth-mans} for random cubic graphs.
Algorithm \ref{alg:algo} is based on a breadth first search which only explores from certain vertices. The algorithm iteratively builds a forest $T$. At each step, a random vertex $v$ is chosen (either a leaf of the current forest or an as-yet unseen vertex) and
the neighbors of $v$ are exposed. If at most one of $v$'s neighbors lies in the current forest $T$, then $v$ may be safely added to $T$ as a full degree vertex. 
In the algorithm below, let $S_G(v)$ represent the star centered at $v$ in $G$, i.e., the graph with vertex set $\set{v}\cup N_G(v)$ and edge set $\{vx\,:\,x\in N(v)\}$. 

   \begin{algorithm}[]
	\KwIn{Connected $r$-regular graph $G=(V,E)$.}
	\KwOut{Tree $T$ with full degree vertices $F$.}
	Select arbitrary $v \in V$\;
	$T = S_G(v)$\;
	$\L=N_G(v)$, $Z_r=V\setminus V_T$\;
	\While{$\L\cup Z_r \neq \emptyset$}
 	{   
 	    \uIf{$\L \neq \emptyset$}{Select $v \in \L$ u.a.r.}
  		\uElse{Select $v \in Z_r$ u.a.r.}
  		\uIf{ $|V_T\cap N_G(v)| \le 1$}{
  		$T = T\cup S_G(v)$\;
  		$F = F\cup \set{v}$\;
  		$X = Z_r\cap N_G(v)$\;
  		Move $X$ from $Z_r$ to $\L$
  		}
  		\uElse{$Z_r = Z_r\setminus N_G(v)$\;
  		$\L = \L\setminus N_G(v)$\;}
  	}
  	Complete the forest $T$ to a spanning tree arbitrarily\;
    \vspace{.2in}
   
	\caption{Full Degree Tree}
	\label{alg:algo}
\end{algorithm}

Our main results, the bounds determined by this algorithm, are as follows.

\begin{theorem}\label{thm:main}When run on a random cubic graph $G\sim G(n,3)$,
Algorithm \ref{alg:algo} w.h.p.\ produces a tree $T$ with at least $0.4591n$ vertices of full degree. Thus $\varphi(G)\geq0.4591n, \lambda(G)\geq0.4591n, \text{ and }\gamma_C(G)\le 0.5409n$. When run on a random $r$-regular graph $G(n,r)$ with $4\le r \le 10$, Algorithm \ref{alg:algo} w.h.p.\ produces a tree with at least $f_r n$ vertices of full degree, where values of $f_r$ are shown in Table \ref{tab:fr}. 
\end{theorem}

\begin{table}[h]
    \centering
    \begin{tabular}{c|c|c}
        $r$ &$f_r$ &$u_r$\\
        \hline
        3 &.4591 &.5000 \\
        4 & .2699 &.3333\\
        5 & .1811& .2500\\
        6 & .1315 & .2000\\
        7 & .1006& .1667\\
        8 & .0799 & .1429\\
        9 & .0652 & .1250 \\
        10& .0545 & .1111
    \end{tabular}
    \caption{Values of $f_r, u_r$ such that $f_r n \le \phi(G(n,r))\le u_r n$ w.h.p. The bounds $f_r$ come from Algorithm \ref{alg:algo} and $u_r = 1/(r-1)$, deterministically from Proposition \ref{prop:simple-bds}.}
    \label{tab:fr}
\end{table}

In the following section, we will use the differential equations method to prove Theorem \ref{thm:main}. We describe the expected one-step changes of several parameters throughout the algorithm in Section \ref{sec:onestep}. Then, in Section \ref{sec:traj}, we will apply a general theorem of Wormald \cite{wormald-phases,wormald} to show concentration of this system of random variables around their expected trajectories in order to complete the proof of Theorem. \ref{thm:main}.

\section{Analysis of Algorithm \ref{alg:algo}}\label{sec:analysis}
In our set up, we make use of the \emph{configuration model}  to analyze our algorithm on $G(n,r)$ (see e.g.\ \cite{wormald-rrg} for more details on the description which follows). Suppose $rn$ is even and consider a set of $rn$ 
\emph{configuration points} partitioned into $n$ labeled \emph{buckets} $v_1,\ldots, v_n$ each of size $r$. A \emph{pairing} of these points is a perfect matching of the configuration points. Given a pairing $P$, we may obtain a multigraph $G(P)$ by contracting each of the buckets to one vertex. It is well known that the restriction of this probability space to simple graphs is $G(n,r)$ and that for fixed $r$, the probability that the pairing generates a simple graph is bounded away from $0$ (independently of $n$). Thus any event which holds w.h.p.\ over the space of random pairings also holds w.h.p.\ for $G(n,r)$. 

We analyze our algorithm by tracking certain parameters throughout the execution of the algorithm. We only reveal partial information about $G(n,r)$ (or more precisely, the pairing) as the algorithm progresses. 
$T$ will represent the current forest being built. 
Each iteration of the while loop will be called a \emph{step}. At each step, we \emph{process} a vertex $v$. When $v$ is processed, we reveal its neighbors in the configuration (some of $v$'s neighbors may already have been revealed). Let $L = L(t)$ represent the vertices which are in $T$ and have $r-1$ unrevealed configuration points. All vertices of $L$ are leaves of $T$, but not all leaves of $T$ are in $L$. 
For $i\in [r]$, let  $Z_i = Z_i(t)$ denote the set of vertices that \emph{are not in} $T$ and have $i$ unrevealed neighbors at time step $t$. $F = F(t)$ represents the set of vertices of degree $r$ in $T$ at time step $t$. Thus, we either process a vertex which is a leaf of $T$ (when $v\in \L$) or a vertex which has never been seen (when $v\in Z_r$). A step is a \emph{success} if at most one of $v$'s neighbors already lies in $T$ (when $v\in \L$, the previously revealed neighbor certainly lies in $T$). In this case, we may add $v$ and its neighbors to $T$.


 We consider our algorithm to have two phases. In Phase 1, we only process vertices from $\L$. We say Phase 2 begins when the first vertex from $Z_r$ is processed. In Phase 2, we process vertices from both $\L$ and $Z_r$. 
In Figure \ref{treediagram}, we illustrate the possibilities when processing a vertex $v$ from $L$ in the case of  a random cubic graph.
In that case, there are three possibilities, shown in Figure \ref{treediagram}: (1) both of $v$'s newly revealed neighbors are not currently in $T$, (2) only one of these neighbors is currently in $T$, or (3) both of these neighbors are currently in $T$. If both of these neighbors are not currently in $T$ (Case 1 in Figure \ref{treediagram}), then we can make $v$ full degree by adding all of its neighbors and incident edges to $T$, i.e. the step is a success. In Cases 2 and 3, adding $v$  and its neighbors to $T$ could potentially create a cycle and so we do not attempt to do so. We note that since $T$ is a forest, having 2 neighbors in $T$ does not guarantee the creation of a cycle, but it is simpler for our analysis to err on the side of caution.

   \begin{figure}[h]
       \centering
       \includegraphics[height=3cm]{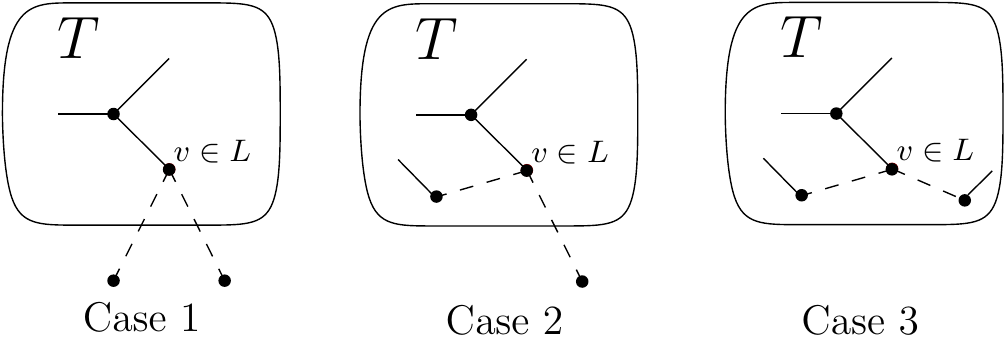}
       \caption{Selecting $v \in L$ and exposing its neighbors}
       \label{treediagram}
   \end{figure}

\subsection{Expected One Step Changes}\label{sec:onestep}

In a common abuse of notation, we refer to $L$, $Z_i$, and $F$ as the sets they represent as well as the sizes of those sets. 
Let 
\[Z := \sum_{i=1}^{r}iZ_i\] so that $Z$ represents the number of unrevealed configuration points corresponding to the $Z_i$'s.
Let $M = M(t)$ represent the number of unrevealed configuration points at time step $t$.
There are two types of operations which we perform. Following notation from \cite{wormald-phases}, let $\Op_1$ denote the operation of processing a vertex from $\L$ and let $\Op_2$ denote the operation of processing a vertex from $Z_r$. Let $\op_t\in \set{\Op_1, \Op_2}$ represent the operation performed at time step $t$ of the algorithm.

\subsubsection{Phase 1}\label{sec:phase1}

Let $\mathcal{F}_t$ represent the revealed part of the configuration model at time $t$. As mentioned before, in Phase 1, we only process vertices from $\L$. For random variable $X = X(t)$, let $\Delta X(t) = X(t+1)-X(t).$ For $i\in [r-1]$,
\begin{align}
\mathbb{E}[\Delta Z_i(t)\, |\, \mathcal{F}_t, \op_t=\Op_1 ] &= (r-1)\cdot\of{-\frac{iZ_i}{M} + \frac{(i+1)Z_{i+1}}{M}\cdot\of{1 - \of{\frac{Z}{M}}^{r-2}}} + O\bfrac 1n \label{eq:zigy}\\
\mathbb{E}[\Delta Z_r(t)\, |\, \mathcal{F}_t, \op_t=\Op_1 ] &= (r-1)\cdot\of{-\frac{rZ_r}{M} } + O\bfrac 1n \label{eq:zrgy}\\
\mathbb{E}[\Delta \L(t)\, |\, \mathcal{F}_t, \op_t=\Op_1 ]&=-1+(r-1)\cdot\of{-\frac{(r-1)\L}{M} + \frac{rZ_r}{M}\cdot\of{\frac{Z}{M}}^{r-2}}+O\bfrac{1}{n} \label{eq:ygy}\\
\mathbb{E}[\Delta F(t)\,|\, \mathcal{F}_t, \op_t=\Op_1]&= \bfrac{Z}{M}^{r-1} + O\bfrac 1n \label{eq:fgy}\\
\mathbb{E}[\Delta M(t)\,|\,\mathcal{F}_t, \op_t=\Op_1]&= -2(r-1).\label{eq:mgy}
\end{align}
To see \eqref{eq:zigy}, note that the vertex $v$ that we are processing has $r-1$ unrevealed configuration points. If any of these points pair with a point in a $Z_i$ vertex, then that vertex is no longer a $Z_i$ vertex. This happens with probability $\frac{iZ_i+O(1)}{M+O(1)} = \frac{iZ_i}{M}+O(1/n)$ where the error term is $O(1/n)$ since we will assume that $M=\Omega(n)$. The error terms in the rest of the explanation are similar and will be ignored. We gain a $Z_i$ vertex if the step is not a success but one of the revealed points pairs with a $Z_{i+1}$ point. The factor $(i+1)Z_{i+1}/M$ represents the probability that a point pairs with a $Z_{i+1}$ point and $1-  \bfrac{Z}{M}^{r-2}$ represents the probability that the other revealed points do not all land in $Z$ (which would mean a success).
In \eqref{eq:ygy}, the $-1$ accounts for the loss of the $\L$ vertex which we are processing. Again, we will lose $\L$ vertices when the revealed points pair to $\L$ vertices.  We gain $\L$ vertices if the step is a success and one of the points pairs with a $Z_r$ vertex. The expected change in $F$ is just the probability that the step is a success which is $\bfrac ZM ^{r-1}$. Finally, note that at each step we reveal $r-1$ pairs from the configuration which explains \eqref{eq:mgy}.


\subsubsection{Phase 2}\label{sec:phase2}
The expected one step changes in Phase 2 when processing a vertex from $\L$ are given by equations \eqref{eq:zigy}--\eqref{eq:mgy}.

The expected one step changes when processing a vertex from $Z_r$ are as follows.
For ease of notation let $P := \bfrac{Z}{M}^{r-1} + (r-1)\cdot \bfrac{Z}{M}^{r-2}\cdot\of{1-\frac{Z}{M}}$. Note that this represents the probability of a success step when performing $\Op_2$ conditional on one revealed point pairing with a $Z$ vertex. For $i \in [r-1]$,
\begin{align}
\mathbb{E}[\Delta Z_i(t)\, |\, \mathcal{F}_t, \op_t=\Op_2 ] &=  r\cdot\of{-\frac{iZ_i}{M} + \frac{(i+1)Z_{i+1}}{M}\cdot\of{1 - P}} + O\bfrac 1n \label{eq:zigz}\\
\mathbb{E}[\Delta Z_r(t)\, |\, \mathcal{F}_t, \op_t=\Op_2 ] &= -1+ r\cdot\of{-\frac{rZ_r}{M}} + O\bfrac 1n \label{eq:zrgz}\\
\mathbb{E}[\Delta \L(t)\, |\, \mathcal{F}_t, \op_t=\Op_2 ]&=r\cdot\of{-\frac{(r-1)\L}{M} + \frac{rZ_r}{M}\cdot P}+O\bfrac{1}{n} \label{eq:ygz}\\
\mathbb{E}[\Delta F(t)\,|\, \mathcal{F}_t, \op_t=\Op_2]&= \bfrac{Z}{M}^r +r\cdot\bfrac{Z}{M}^{r-1}\cdot\of{1-\frac{Z}{M}} + O\bfrac 1n \label{eq:fgz}\\
\mathbb{E}[\Delta M(t)\,|\,\mathcal{F}_t, \op_t=\Op_2]&= -2r\label{eq:mgz}
\end{align}

The explanations for these are similar to those for $\Op_1$. The main differences are that (i) we now reveal $r$ pairs from the configuration model rather than $r-1$, (ii)  a step is still a success if one of the revealed points pairs with a $T$ (i.e. ``non-$Z$'') vertex.

\subsection{Trajectories from differential equations}\label{sec:traj}
For ease of notation, we let $a=r+3$ and set $\tbf{Z} = (Z_1,\ldots,Z_{a}) = (Z_1,\ldots, Z_r, L, F, M)$.
For $i\in[a]$ and $j\in [2]$, we define 
$f_{i,j} (t/n, Z_1(t)/n, \ldots, Z_a(t)/n)$ to be the expression given on the right hand side shown in \eqref{eq:zigy}-\eqref{eq:mgz} for $\mathbb{E}[\Delta Z_i(t)\, |\, \mathcal{F}_t, \op_t=\Op_j ]$ ignoring the $O(1/n)$ terms so as to remove the dependence on $n$. So for example, letting $x=t/n$, $\tbf{z}=(z_1,\ldots, z_a)$ and $z = \sum_{i=1}^r iz_i$ where $z_i(x) = Z_i(t)/n$, we have 
$f_{i, 1}(x,\tbf{z}) = (r-1)\of{-\frac{iz_i}{z_{r+3}} + \frac{(i+1)z_{i+1}}{z_{r+3}}\cdot\of{1 - \bfrac{z}{z_{r+3}}^{r-2}}}$ for $i\in[r-1]$.

The details of the following differential equations method have been omitted, but they are by now standard (see for example \cite{duckworth-mans, duckworth-zito, wormald-rrg,wormald-phases}). In Phase 1, we have an ordinary application of the differential equations method. In Phase 2, we have a prioritized algorithm which performs a mixture of 2 types of steps. The results of Wormald essentially say that we may instead analyze a deprioritized algorithm which selects vertices according to a pre-determined probability function (which in effect blends the two types of steps appropriately). We note that our functions $f_{i,j}$ are well behaved (e.g.\ they have continuous and bounded derivatives) as long as $z_{r+3} = M/n$ stays bounded away from 0. As our numerical solutions show\footnote{A Maple worksheet which can be used to verify the claimed results can be found at \url{https://msuweb.montclair.edu/~bald/research.html}}, this is the case for all values of $r$ considered in Table \ref{tab:fr}. 

We now describe the blending of the steps for Phase 2 as described in \cite{wormald-phases}.
Suppose in Phase 2, an $\Op_2$ creates, in expectation, $\alpha$ many vertices of $L$ and suppose that performing an $\Op_1$ decreases the number of $L$ vertices, in expectation, by $\tau$. Then we would expect an $\Op_2$ to be followed by $\alpha/\tau$ many $\Op_1$ steps. Then in Phase 2, we would expect the proportion of $\Op_2$ steps to be $1/(1 + \alpha/\tau) = \tau/(\tau+\alpha)$ and the proportion of $\Op_1$ steps to be $\alpha/(\tau + \alpha)$. 
Let $x=t/n$ and let $\tbf{z} = (z_1,\ldots, z_a)$. Then (recalling that $L$ is now represented by $Z_{r+1}$) the asymptotic values of $\alpha$ and $\tau$ are given by 
\[\alpha = f_{r+1, 2}(x, \tbf{z}),\qquad \tau = -f_{r+1, 1}(x, \tbf{z}).\]
We let \[p:= \frac{\alpha}{\tau+\alpha}\] and 
\[F(x, \tbf{z}, i, k) = 
\begin{cases}
f_{i,1}(x, \tbf{z}) &\textrm{ if } k=1\\
p\cdot f_{i,1}(x,\tbf{z}) + (1-p)\cdot f_{i,2}(x, \tbf{z}) &\textrm{ if } k=2
\end{cases}\]
Then for Phase $k\in [2]$, and $i\in [a]$, we let 
\begin{equation}\label{eq:sys}
    \frac{d\tz_i}{dx} = F(x,\ttz, i, k)
\end{equation} 
with initial conditions for Phase 1 given by $\tz_i(0)= 0$ for $i\in \set{1,\ldots, r-1,r+1, r+2}$, $\tz_r(0)=1$, $\tz_{r+3}(0) = r$. 
Numerical solutions of the Phase 1 system show that in this phase, $\tz_{r+1}$ increases and then decreases until it hits zero at which time Phase 2 begins. Let $\rho_1^r > 0$ be the first time when $\tz_{r+1}(\rho_1^r)=0.$ The numerical solutions show that all other tracked variables are bounded away from 0 at time $\rho_1^r$. The initial conditions for Phase 2 are given by the final values of Phase 1, i.e. $\tz_i(\rho_1^r)$ for all $i\in [a]$. 
 Numerical solutions of the Phase 2 system then imply that Phase 2 ends at a time $\rho_2 = \rho_2^r$, when $\tz_r(\rho_2) = 0$. A more detailed explanation of the solutions in the cases $r=3$ and $4$ can be found in the Appendix. The conclusion of the differential equations method is that $Z_i(t) = n\tz(t/n) + o(n)$ for all $i$ and for all $0\le t \le \rho_2^r n$.
 The number of full degree vertices can then be represented by $f_rn = \tz_{r+2}(\rho_2^r)\cdot n$.

\section{Conclusion}\label{sec:conc}
We point out that some of the functions in the Phase 1 system can be solved for analytically. In particular, we have, using the inital conditions $\tz_{r+3}(0)=r$ and $\tz_r(0)=1$, that in Phase 1,
\[\tz_{r+3}(x) = r - 2(r-1)x\]
and hence
\[\tz_r(x) = \of{1 - \frac{2(r-1)}{r}x}^{r/2}.\]
Unfortunately, the equations for the other variables depend on $z = \sum_i iz_i$ which we don't currently see how to deal with.

It is an intriguing open problem to determine whether random regular graphs contain full degree spanning trees with an optimal (or asymptotically optimal) number of full degree vertices. 
\begin{problem}
Does $G\sim G(n,r)$ satisfy $\phi(G)  = \frac{n}{r-1}(1+o(1))$?
\end{problem}
Perhaps applying the second moment method to a particular tree with only vertices of degree $r$ and $1$ could be used to show this. We note that as a function of $r$, our $f_r$ seem to decay at a rate faster than $1/r$, so it seems unlikely that a simple modification of our algorithm will succeed in proving this. 

The lower bound in Proposition \ref{prop:simple-bds} can be thought of as arising from a greedy algorithm which removes a vertex, and its first and second neighborhoods at each step. By analyzing a slightly more sophisticated algorithm (see  Lemma 5.2 in \cite{balschudrich})
 which allows the neighborhoods of the chosen vertices to overlap in one vertex, one can prove the lower bound 
 $\phi(G) \ge \frac{2}{\Delta^2 + \Delta + 2}\cdot n = \of{1 + o_\Delta(1)}\frac{2}{\Delta^2} \cdot n$. 
 This lower bound is almost tight (up to a constant factor) as can be seen by the following construction which essentially appears in  \cite{bhatia}. Let $G = (K_{\Delta/2} \square K_{\Delta/2})\square C_{4/\Delta^2}$. Here, one can take at most one vertex to be of full degree from each copy of $K_{\Delta/2} \square K_{\Delta/2}$ and so $\phi(G)\le 4/\Delta^2$.  It would be interesting to improve this factor of 2. 
\begin{problem}
    Determine the best possible deterministic lower bound for $\phi(G)$ for all connected graphs with $\Delta(G) = \Delta$.
\end{problem}

\section*{Acknowledgments}
We would like to thank Patrick Bennett for helpful discussions about the topic.

\printbibliography

\appendix
\newpage 

\section{Discussion of $r=3,4$}

\begin{figure}
    \centering
    \begin{minipage}{0.45\textwidth}
        \centering
        \includegraphics[width=0.9\textwidth]{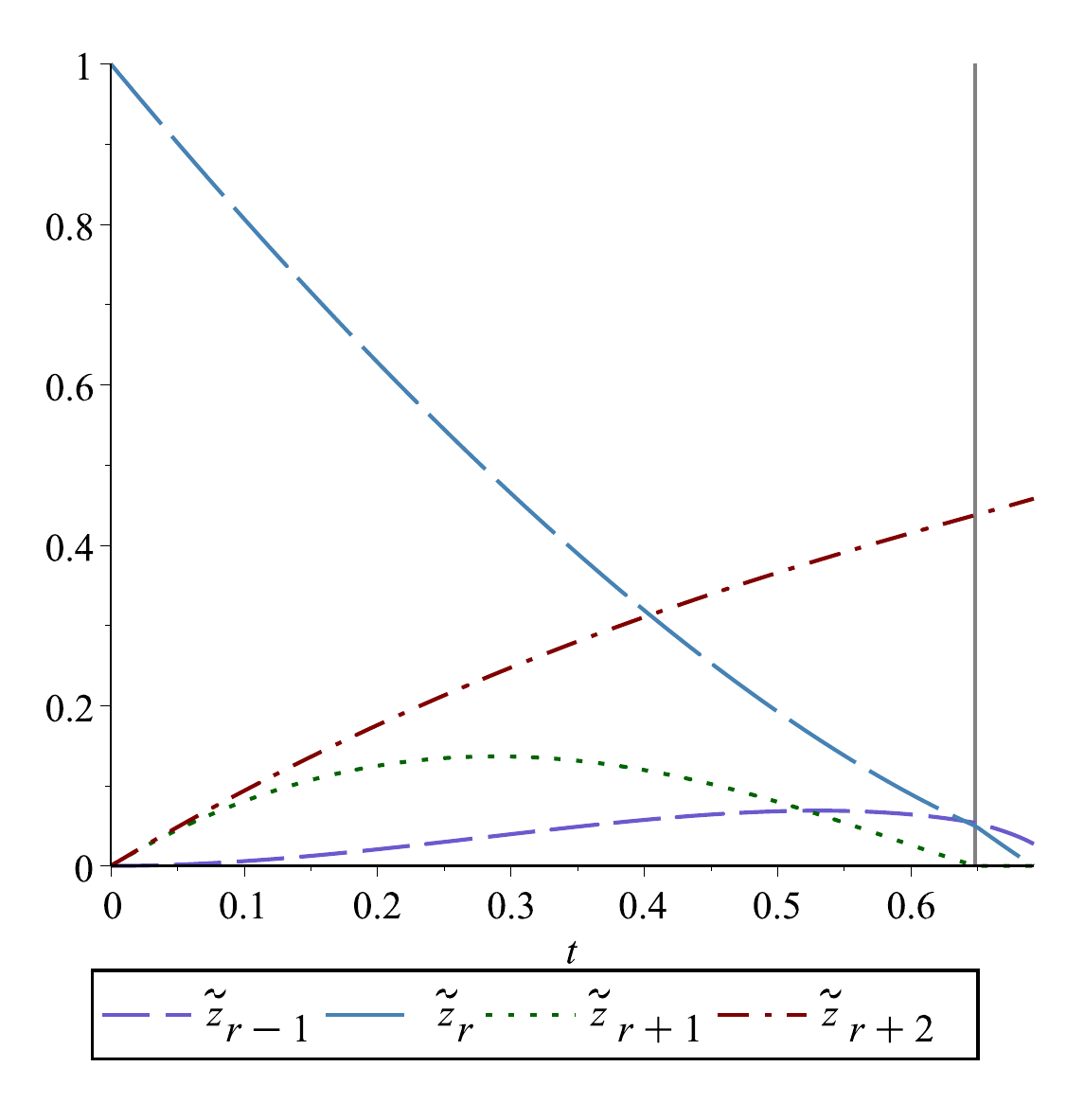} 
        
        \caption{Trajectories for $r=3$}\label{fig:r3}
    \end{minipage}\hfill
    \begin{minipage}{0.45\textwidth}
        \centering
        \includegraphics[width=0.9\textwidth]{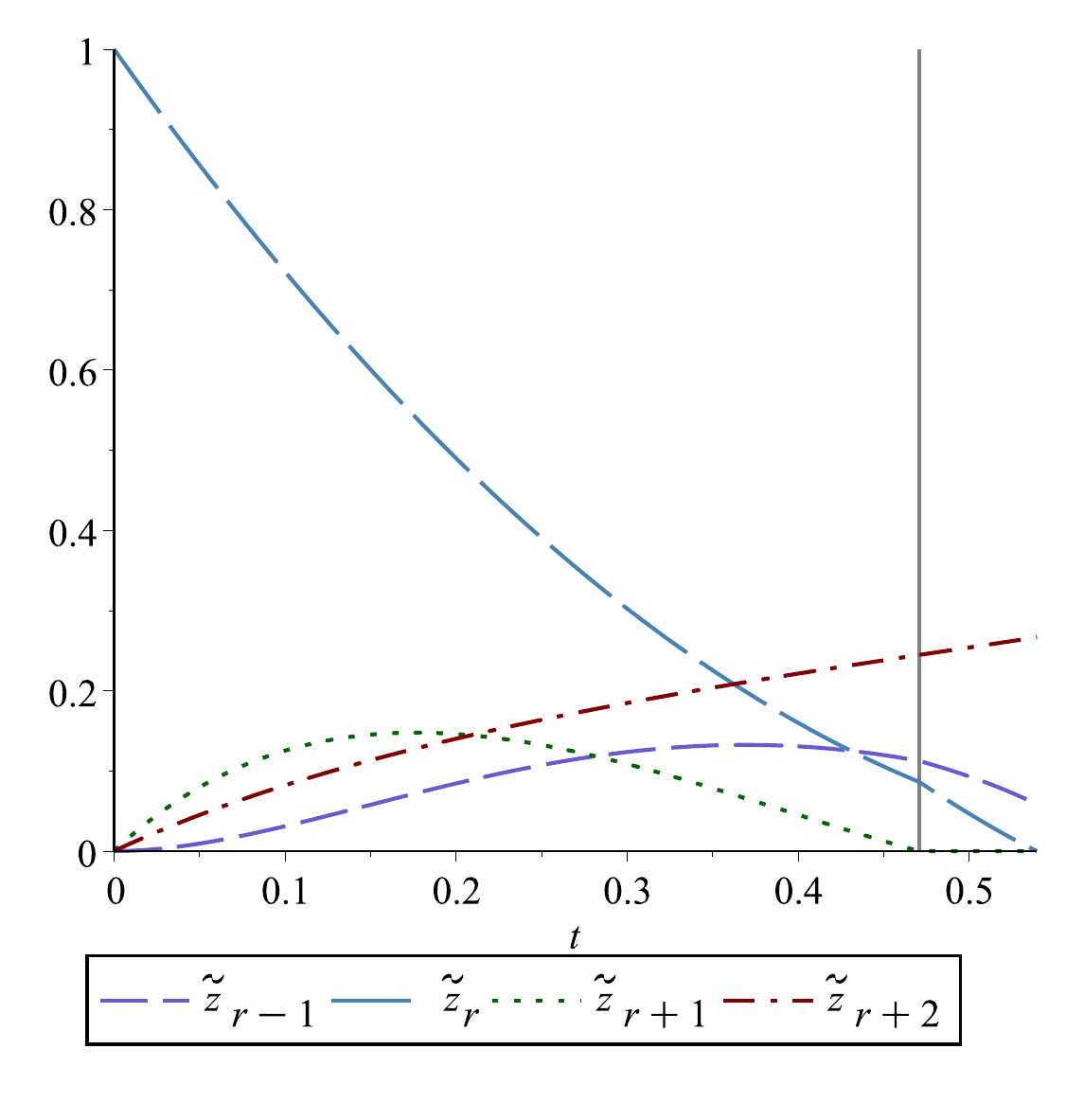} 
        
        \caption{Trajectories for $r=4$}\label{fig:r4}
    \end{minipage}
\end{figure}

\subsection{$r=3$}
In Figure \ref{fig:r3} we see some of the solutions to the system given in \eqref{eq:sys} for $G\sim G(n,3)$. Phase 1 ends at the vertical line at time $\rho_1 \approx 0.6485$. The initial conditions for Phase 2 are then given by
$\tz_1(\rho_1) \approx 0.0193$, $\tz_2(\rho_1)\approx 0.0536$, $\tz_3(\rho_1)\approx 0.0498$, $\tz_4(\rho_1)=0$, $\tz_5(\rho_1)\approx 0.4375$, $\tz_6(\rho_1)\approx 0.4060$. Thus at the end of Phase 1, the algorithm has found a forest with $\approx 0.4375 n$ vertices of full degree. We note that this already implies that w.h.p.\ $\lambda(G) \ge 0.4375 n$ and $\gamma_C(G) \le 0.5625 n$, an improvement over the best known bounds given in \cite{duckworth} for these problems. Phase 2 then ends when $\tz_3 =0$ at time $\rho_2 \approx 0.6922$. As one can see, $\tz_2$ (and hence $\tz_{r+3}$) is bounded away from 0 at $\rho_2$.

\subsection{$r=4$}
In Figure \ref{fig:r4} we see some of the solutions to the system given in \eqref{eq:sys} for $G\sim G(n,4)$. Phase 1 ends at the vertical line at time $\rho_1 \approx 0.4707$. The initial conditions for Phase 2 are then given by
$\tz_1(\rho_1) \approx 0.0119$, $\tz_2(\rho_1)\approx 0.0548$, $\tz_3(\rho_1)\approx 0.1124$, $\tz_4(\rho_1)\approx 0.0864$, $\tz_5(\rho_1)= 0$, $\tz_6(\rho_1)\approx 0.2445$, $\tz_7(\rho_1)\approx 1.1757$.   Phase 2 then ends when $\tz_4 =0$ at time $\rho_2 \approx 0.5397$.
   
\end{document}